\documentclass[12pt]{article}
\usepackage{amssymb, amsmath}
\usepackage{amsthm,array,hyperref,caption,subcaption,float,longtable,multirow,accents}
\usepackage{bm}
\usepackage{url}

\usepackage[margin=1 in]{geometry}
\usepackage{enumerate}
\allowdisplaybreaks[1]
\numberwithin{equation}{section}

\newlength{\dhatheight}
\newcommand{\doublehat}[1]{%
	\settoheight{\dhatheight}{\ensuremath{\hat{#1}}}%
	\addtolength{\dhatheight}{-0.35ex}%
	\hat{\vphantom{\rule{1pt}{\dhatheight}}%
		\smash{\hat{#1}}}}
\newcolumntype{M}[1]{>{\centering\arraybackslash}m{#1}}

\DeclareMathOperator{\adj}{adj}

\newtheorem{example}{Example}[section]
\newtheorem{thm}{Theorem}[section]
\newtheorem{cor}{Corollary}[section]
\newtheorem{note}{Note}[section]
\newtheorem{pro}{Proposition}[section]

\newtheorem{rem}{Remark}[section]
\newtheorem{defn}{Definition}[section]

\newtheorem{lemma}{Lemma}[section]
\usepackage{graphicx}
\begin{document}
	\vspace{-2cm}
	\markboth{S. P. Leka Amruthavarshini and R. Rajkumar}{Adjacency-diametrical matrix of a graph}
	\title{\LARGE\bf Adjacency-diametrical matrix of a graph}
	\author { S. P. Leka Amruthavarshini\footnote{Email: lekaamrutha@gmail.com}~ and  R. Rajkumar\footnote{Email: rrajmaths@yahoo.co.in (Corresponding author)} 
		\\ \small \it Department of Mathematics,
		\small \it The Gandhigram Rural Institute (Deemed to be University),\\
		\small \it Gandhigram-624 302, Tamil Nadu, India.
	}

	\date{}
	\maketitle
	\begin{abstract}
		The adjacency-diametrical matrix (AD matrix) of a connected graph $G$ with diameter $d$, denoted by $AD(G)$, is the matrix indexed by the vertices of $G$ in which the $(i,j)$-entry of $AD(G)$ is $1$ if $d_G(v_i,v_j)=1$, is  $d$ if $d_G(v_i,v_j)=d$, and $0$ otherwise, where $d_G(v_i,v_j)$ denotes the distance between the vertices $v_i$ and $v_j$ in $G$. We determine the spectrum of the AD matrix for paths, cycles, and double star graphs and obtain its determinant for a connected graph. We characterize a class of bipartite graphs using the coefficients of the characteristic polynomial and the eigenvalues of the AD matrix. We establish bounds relating the eigenvalues of the AD matrix to various graph invariants, and we determine the spectrum of the AD matrix for graphs formed by the join, lexicographic product, and Cartesian product operations under certain conditions on the constituent graphs.\\

		
		
	\end{abstract}
	
	\section{Introduction}

We consider only undirected, simple and finite
graphs.  Let $G$ be a  graph  with vertex set $V(G)=\{v_1, v_2, \ldots, v_n\}$ and edge set $E(G)$. If $v_i$ and $v_j$ are adjacent in $G$, we write $v_i\sim v_j$. The adjacency matrix $A(G)$ is the $0$–$1$ matrix whose rows and columns correspond to the vertices of $G$. For $i,j\in\{1,2,
\ldots,n\}$, the $(i,j)$-entry  is $1$ exactly when $v_i$ and $v_j$ are adjacent. This definition extends to a weighted graph with an edge–weight function $w: E(G)\to\mathbb{R}^+$ by setting the $(i,j)$-entry of $A(G)$ to be $w(e)$ whenever $e=\{v_i, v_j\}\in E(G)$. 	Suppose $G$ is a connected graph with diameter $d$. For each integer $k$ with $1 \leq k \leq d$, the $k$-th adjacency matrix $A_k(G)$ is the $0-1$ matrix whose $(i,j)$-entry entry is $1$ if and only if $d_G(v_i,v_j)=k$, where $d_G(v_i,v_j)$ denotes the distance between the vertices $v_i$
and $v_j$, i.e., the length of a shortest path between $v_i$ and $v_j$. Note that $A_1(G)=A(G)$. The distance matrix $D(G)$ of $G$ is the  matrix  indexed by the vertices of $G$, where for $i,j \in \{1,2,\ldots,n\}$, the $(i,j)$-entry of $D(G)$ is $d_{G}(v_i,v_j)$. It is evident that $D(G)=\sum_{k=1}^dkA_k(G)$.

A large literature is devoted to the study of adjacency matrix of a graphs and its spectrum. We refer the reader to \cite{bapat2010graphs,cvetkovic2010introduction} for more details. The distance matrix is another significant matrix, providing valuable insights into the metric and spectral characteristics of a graph. It has a wide range of applications, including in chemistry, molecular biology, psychology, and various other scientific fields. For more details of this matrix and its variants, we refer the reader to~\cite{aouchiche2014survey,aouchiche2013twolaplacians,hogben2022spectra,stevanovic2012spectral}.	
Recently,	De Ville introduced the notion of a generalized distance matrix in~\cite{deville2022thegeneralized}. 
	The generalized distance matrix of a connected $G$ with diameter $d$ is defined as the matrix $\mathcal{M}(f; G)$, whose entries are given by $
	\mathcal{M}(f; G)_{v_i,v_j} = f(d_{G}(v_i,v_j)),
$ where $f$ is a real-valued function on $\{0,1,\ldots, d\}$. It is shown that several familiar matrices, including the adjacency matrix, Laplacian matrix, and distance matrix, appear as special instances of this construction. Moreover, the paper presents applications of the spectral properties of this matrix to competition models in ecology and to rapidly mixing Markov chains.
	
	Now, we define the  \textit{adjacency-diametrical matrix} (or \textit{AD matrix}) of a connected graph $G$ with diameter $d$. It is denoted by $AD(G)$, and is given by $$AD(G)=A(G)+dA_d(G).$$ Equivalently, for $i,j=1,2,\ldots,n$, the $(i,j)$-entry of $AD(G)$ is $1$ if $d_G(v_i,v_j)=1$; is $d$ if $d_G(v_i,v_j)=d$, and $0$ otherwise. Note that $AD(G)$ is a special case of the generalized distance matrix obtained by taking the function $f$ satisfying $f(1)=1$, $f(d)=d$, and $f(i)=0$ for every $i=2,3,\dots,d-1$. 
	The AD matrix is introduced to explore how a graph’s diameter relates to its other invariants and to examine how the matrix’s eigenvalues, eigenvectors, and characteristic polynomial interact with the diameter and related structural properties of the graph. 
	
		Note that $d=1$ if and only if $AD(G)=A(G)=D(G)$ and $G=K_n$. If $d=2$, then $AD(G)=D(G)$, and conversely $AD(G)=D(G)$ if and only if $d \in \{1,2\}$.

	We associate to $G$ a weighted graph $\mathcal{G}_G$ defined as follows: its vertex set is the same as that of $G$, and two vertices $v_i$ and $v_j$ are joined by an edge $e_{ij}$ of weight $1$ or $d$ precisely when $d_G(v_i,v_j)=1$ or $d_G(v_i,v_j)=d$ respectively. Clearly, $A(\mathcal{G}_G)=AD(G)$.

	We now recall some definitions and notations. In a graph $G$, the degree $d(v)$ of its vertex $v$  is the number of edges incident with it. If all vertex in $G$ have the degree $r$, then $G$ is called an $r$-regular graph. In a weighted graph $G$, the weighted degree of its vertex $v$ is defined as the sum of the weights of the edges incident with $v$.  The weighted graph $G$ is said to be weighted regular of degree $r$ or weighted $r$-regular, if every vertex of $G$ has weighted degree $r$. As usual $P_n$, $C_n$, and $K_n$ denote the path, cycle and complete graph each on $n$ vertices, respectively. The star graph on $n$ vertices is denoted by $S_{n}$. The double star graph  $S_{n_1,n_2}$ constructed by taking two disjoint star graphs $S_{n_1}$ and $S_{n_2}$ and then joining their centers by a single edge. A graph is called bipartite if its vertex set $V$ can be partitioned into two subsets $V_1$ and $V_2$ such that every edge of $G$ connects a vertex from $V_1$ to a vertex from $V_2$. A graph $G$ is said to be planar if it can be drawn in a plane without intersecting edges. 
	Two vertices in a connected graph $G$ are said to be antipodal, if the distance between these vertices is equal to the diameter of $G$.  The chromatic number $\chi(G)$ of a graph $G$ is  the minimum number of colors needed for coloring the vertices of $G$ such that no two adjacent vertices receives the same color.
The clique number $\omega(G)$ of a graph $G$  is the number of vertices in the maximal complete subgraph in $G$. 	
A distance-regular graph is a regular graph in which, for any vertices $u$ and $v$, the number of vertices at distance $j$ from $u$ and distance $k$ from $v$ depends only on $j$, $k$, and the distance between $u$ and $v$.

	Let $M$ be an $n \times n$ matrix. The trace of $M$ is written as $\mbox{tr}(M)$, and its characteristic polynomial is denoted by $\Phi_M(x)$. The eigenvalues of $M$ are denoted by $\lambda_i(M)$ for $i = 1,2,\dots,n$, and the multiset of these eigenvalues, called the spectrum of $M$, is written as $\mbox{Spec}(M)$.	For two real $n \times n$ matrices $A_1$ and $A_2$, we say that $\lambda_1$ and $\lambda_2$ are co-eigenvalues of $A_1$ and $A_2$ if there exists a vector $X \in \mathbb{R}^n$ such that
	$A_i X = \lambda_i X$ for $i = 1,2$.	We denote by $I_n$ the $n \times n$ identity matrix, by $\mathbf{1}$ the $n \times 1$ all-ones vector, and by $\mathbf{0}$ the zero matrix.

	 The rest of the paper is structured as follows: In Section~\ref{s4}, we compute the  spectrum of the AD matrix of  paths, cycles and double star graphs. In Section~\ref{s5}, we obtain the determinant of the AD matrix of a connected graph.  In addition, a special class of bipartite graphs, called diametrical bipartite graphs, and characterize this class using the coefficients of the characteristic polynomial and the eigenvalues of the AD matrix. In Section~\ref{s6}, we establish bounds relating the eigenvalues of the AD matrix to various graph invariants. In Section~\ref{s7}, we determine the spectrum of the AD matrix formed by the join, lexicographic product, and Cartesian product operations.

	\section{Spectrum of the AD matrix of some graphs}\label{s4}
		We refer to the spectrum of the AD matrix of a graph $G$ as the \textit{adjacency-diametrical spectrum} (or \textit{AD-spectrum}) of $G$. We begin by computing the characteristic polynomial of the AD matrix of the path graph $P_n$.
	\begin{thm}
		\begin{enumerate}[(i)]
			\item  The characteristic polynomial of $AD(P_2)$ is $x^2-1$. 
			\item  For $n \geq 3 $, the characteristic polynomial of $AD(P_n)$ is 
			\[  \Phi_{n}(x) -(n-1)^2\Phi_{n-2}(x)+2(1-n),\]
			
			where $ \Phi_{n}(x)=\overset{n}{\underset{k=1}{\prod}} \left(x-2 \cos \frac{\pi k}{n+1}\right)$.
		\end{enumerate}
	\end{thm}
	
	\begin{proof}
		(i). Since $P_2 = K_2$, the result is straight forward.
		
		(ii). For $n\geq 3$, the diameter of $P_n$ is $n-1$. By a suitable relabeling of the vertices of $P_n$, we get
		\[AD(P_n)=
		\begin{bmatrix}
			0 & n-1 & 1 & \dots & 0 & 0 \\
			n-1 & 0 & 0 &\dots & 0 & 1 \\
			1 & 0 & 0 & \dots & 0 & 0 \\
			\vdots & \vdots &  \vdots  & \ddots & \vdots& \vdots \\
			0 & 0 & 0 & \dots & 0 & 1\\
			0 & 1 & 0 & \dots & 1 & 0 
		\end{bmatrix},
		\]		
		and the characteristic polynomial of $AD(P_n)$ is given by
		\begin{align}
			\Phi_{AD(P_n)}(x)&=\det \left(xI_n-AD(P_n)\right) \nonumber\\
			&=\det\left(
			\begin{bmatrix}
				B & C\\
				C^T & xI_{n-2}-A(P_{n-2}) 
			\end{bmatrix}\right) \nonumber\\
			&=\det(xI_{n-2}-A(P_{n-2})) \times \det\left( B - C (xI_{n-2}-A(P_{n-2}))^{-1} C^T \right),\label{e4.1}
		\end{align}	
		where $B=
		\begin{bmatrix}
			x & 1-n  \\
			1-n & x
		\end{bmatrix}$ and $C=
		\begin{bmatrix}
			-1 & 0 & \cdots & 0 & 0 \\
			0 & 0 & \cdots & 0 & -1 \\
		\end{bmatrix}.$\\
		
		For convenience, we write $\Phi_n (x) := \Phi_{A(P_{n})}(x)$ for all $n$ and it is known~(cf.~\cite[Theorem 3.7]{bapat2010graphs}) that
		\begin{equation}\label{e4.2}
			\Phi_n(x)= \overset{n}{\underset{k=1}{\prod}} \left(x-2 \cos \tfrac{\pi k}{n+1}\right).
		\end{equation}
				
		Also we have \begin{align}
				\det & \left( B - C (xI_{n-2}-A(P_{n-2}))^{-1} C^T \right) \nonumber\\
			& = \det \left(B-  \left(C~ \tfrac{1}{\Phi_{n-2}(x)} \adj(xI_{n-2}-A(P_{n-2}))  C^{T}\right)\right)\nonumber \\
			& = \det\left( B - \tfrac{1}{\Phi_{n-2}(x)} \begin{bmatrix}
				\Phi_{n-3}(x)  & 1  \\
				1 & \Phi_{n-3}(x)
			\end{bmatrix}\right) \nonumber \\
				& = \det\left(
			\begin{bmatrix}
				x-\tfrac{\Phi_{n-3}(x)}{\Phi_{n-2}(x)}& (1-n) -\tfrac{1}{\Phi_{n-2}(x)} \\
				(1-n) -\tfrac{1}{\Phi_{n-2}(x)} & x-\tfrac{\Phi_{n-3}(x)}{\Phi_{n-2}(x)}
			\end{bmatrix}\right) \nonumber\\
			&= 	\left(x^2-n^2+2n-1\right)- \tfrac{2x(\Phi_{n-3}(x))}{\Phi_{n-2}(x)} 
			+\tfrac{(\Phi_{n-3}(x))^2-1}{(\Phi_{n-2}(x))^2}+\tfrac{2(1-n)}{\Phi_{n-2}(x)}. \label{e4.3}
			\end{align}
		
 	Applying~\eqref{e4.2} and~\eqref{e4.3} in~\eqref{e4.1} we get 
		\begin{equation}\label{e2.4}
			\Phi_{AD(P_n)}(x)=\left(x^2-n^2+2n-1\right) \Phi_{n-2}(x) - 2x \Phi_{n-3}(x) 
			+\tfrac{(\Phi_{n-3}(x))^2-1}{\Phi_{n-2}(x)}+2(1-n).
		\end{equation} 
			
	It is known that the characteristic polynomial of the adjacency matrix of the path
	$P_n$ and  the Chebyshev polynomial of the second kind are related by $\Phi_{n}(x)=U_n(\tfrac{x}{2})$, see e.g.,~\cite{cvetkovic2010introduction}. If $t=\tfrac{x}{2}$, then $\Phi_{n-2}(x)=U_{n-2}(t)$, $\Phi_{n-3}(x)=U_{n-3}(t)$.
		
		Since $U_n(t)$ satisfies the identity $U^2_k(t)-U_{k-1}(t)U_{k+1}(t)=1$~\cite[Eqn. 3.21]{voll2010cassini}, for $k=n-3$ we obtain $\Phi^2_{n-3}(x)-\Phi_{n-4}(x)\Phi_{n-2}(x)=1$. This implies that the term $\tfrac{(\Phi_{n-3}(x))^2-1}{(\Phi_{n-2}(x))^2}$ equals $\Phi_{n-4}(x)$. Then~\eqref{e2.4} becomes,
		\begin{equation}\label{e2.5}
			\Phi_{AD(P_n)}(x)=\left(x^2-n^2+2n-1\right) \Phi_{n-2}(x) - 2x \Phi_{n-3}(x)+\Phi_{n-4}(x)+2(1-n)
		\end{equation}
		
		Also $U_n(t)$ satisfying the recurrence relation $U_{k}(t)=2tU_{k-1}(t)-U_{k-2}(t)$~\cite[Eqns. 1.6a, 1.6b]{mason2002chebyshev}, with $U_0(t)=1, U_1(t)=2t$. Then for $k=n-2$, $k=n-1$ and $k=n$ we get $\Phi_{n-4}(x)=x\Phi_{n-3}(x)-\Phi_{n-2}(x)$, $\Phi_{n-3}(x)=x\Phi_{n-2}(x)-\Phi_{n-1}(x)$ and $\Phi_{n}(x)=x\Phi_{n-1}(x)-\Phi_{n-2}(x)$ respectively. Substituting these in~\eqref{e2.5}, we obtain
		\begin{align}
			\Phi_{AD(P_n)}(x)&=\left(x^2-n^2+2n-1\right) \Phi_{n-2}(x) - 2x \Phi_{n-3}(x)+x\Phi_{n-3}(x)-\Phi_{n-2}(x)+2(1-n) \nonumber\\
			&=\left(x^2-n^2+2n-2\right) \Phi_{n-2}(x) - x \Phi_{n-3}(x)+2(1-n)\nonumber\\
			&=\left(x^2-n^2+2n-2\right) \Phi_{n-2}(x) - x^2 \Phi_{n-2}(x)+x\Phi_{n-1}(x)+2(1-n)\nonumber\\
			&=-\Phi_{n-2}(x) -(n-1)^2\Phi_{n-2}(x)+x\Phi_{n-1}(x)+2(1-n)\nonumber\\
			&=\Phi_{n}(x) -(n-1)^2\Phi_{n-2}(x)+2(1-n).\nonumber
		\end{align}
		This completes the proof.
	\end{proof} 
 The AD-spectrum of the cycle graph is determined in the following result.
	\begin{thm}
		For $n\geq 3$, we have
		\[
		\mbox{Spec}(AD(C_n)) =
		\begin{cases}
		\left\{ 2 \cos \frac{2 \pi k}{n}+ \frac{n}{2} \cos \pi k   \,:\, k = 1,2, \dots, n \right\}, & \text{if } n \text{~is even}; \\
			\left\{ 2 \cos  \frac{2 \pi k}{n} + (n-1) \cos \pi k\; \cos \frac{\pi k}{n}:\, k = 1,2, \dots, n \right\}, & \text{if } n \text{ is odd}.
		\end{cases}
		\]
	\end{thm}
	\begin{proof}
		Note that the diameter of $C_n$ is $\frac{n}{2}$ when $n$ is even and $\frac{n-1}{2}$ when $n$ is odd. Moreover, the matrix $AD(C_n)$ is circulant: if $n$ is odd, the nonzero entries in its first row are $c_1 = c_{n-1} = 1$ and $c_{\tfrac{n}{2}} = \tfrac{n}{2}$; if $n$ is even, the nonzero entries in its first row are $c_1 = c_{n-1} = 1$ and $c_{\tfrac{n-1}{2}} = c_{\tfrac{n+1}{2}} = \tfrac{n-1}{2}$. 
		It is known that for an $n \times n$ circulant matrix $C$ with first row $(c_0,c_1,\dots,c_{n-1})$, the eigenvalue $\lambda_k$ of $C$ is given by 
		\[\lambda_k=\overset{n-1}{\underset{j=0}{\sum}}c_j \, \omega^{jk}_n, \qquad  j=0,1,\dots,n-1,\] where $\omega_n = e^\frac{2 \pi i}{n}$ is a primitive $n$-th root of unity.
		
			Suppose $n$ is even. 	The eigenvalues $\lambda_k$ of $AD(C_n)$ for $k=0,1\dots,n-1$ are given by 
		\[	
		\lambda_k =	\omega^k +(\omega^k)^{n-1}+\tfrac{n}{2}(\omega^k)^{\tfrac{n}{2}}\\
			 = e^\frac{2 \pi i k}{n} +e^{-\frac{2 \pi i k}{n}}+\tfrac{n}{2} (e^\frac{2 \pi i k}{n})^{\frac{n}{2}}\\
			 =2 \cos \tfrac{2 \pi k}{n}+ \tfrac{n}{2} \cos \pi k.\]

		Suppose $n$ is odd.   The  eigenvalues $\lambda_k$ of $AD(C_n)$ for $k=0,1\dots,n-1$ are given by  
	  	\[
		\begin{aligned}
			\lambda_k&=\omega^k +(\omega^k)^{n-1}+\left(\tfrac{n-1}{2}\right)((\omega^k)^{\frac{n-1}{2}}+(\omega^k)^{\frac{n+1}{2}})\\
			& = \omega^k +\omega^{-k}+\left(\tfrac{n-1}{2}\right) \left\{(\omega^k)^{\tfrac{n}{2}} \left((\omega^k)^{\frac{-1}{2}}+(\omega^k)^{\frac{1}{2}}\right) \right\}\\
			& = 2 \cos \tfrac{2 \pi k}{n}+\left(\tfrac{n-1}{2}\right) \left\{e^{\pi i k} \left(e^{-\tfrac{\pi i k}{n}}+e^{\tfrac{\pi i k}{n}}\right)\right\}\\
			& = 2 \cos \tfrac{2 \pi k}{n}+ (n-1) \left\{\cos \pi k \,\cos \tfrac{\pi k}{n}\right\}
		\end{aligned}
		\]
	 and the proof is complete. 
	\end{proof}
	The characteristic polynomial of the AD matrix of the double star graph $S_{n_1,n_2}$ is described in the following result.
	\begin{thm}
	The characteristic polynomial of $AD(S_{n_1,n_2})$ is 
		\[x^{n_1+n_2-4}\left(x^4-\left(9n_1n_2-8n_1-8n_2+8\right)x^2+4\left(n_1n_2-n_1-n_2+1\right)\right).\]
	\end{thm}
	\begin{proof}
	By suitably arranging the vertices of $S_{n_1,n_2}$, it can be seen that
		\[AD(S_{n_1,n_2})=\begin{bmatrix}
		0 & 1 & J_{1 \times (n_1-1)} & \bm{0}_{1 \times (n_2-1)}\\
		1 & 0 & \bm{0}_{1 \times (n_1 - 1)} & J_{1 \times (n_2 -1)}\\

		J_{(n_1 - 1) \times 1} & \bm{0}_{(n_1-1) \times 1} & \bm{0}_{(n_1-1) \times (n_1-1)} & 3J_{(n_1-1) \times (n_2-1)}\\
		\bm{0}_{(n_2-1) \times 1} & J_{(n_2-1) \times 1} & 3J_{(n_2-1) \times (n_1-1)} & \bm{0}_{(n_2-1) \times (n_2-1)}
		\end{bmatrix}.\]

		Since each row of $\bm{0}_{(n_1-1) \times (n_1-1)}$ (resp. $\bm{0}_{(n_2-1) \times (n_2-1)}$) has equal row sum $0$, the vector $\bm{1}_{(n_1-1) \times 1}$ (resp. $\bm{1}_{(n_2-1) \times 1}$) is an eigenvector corresponding to the eigenvalue $0$. Assume that the remaining eigenvectors $X_i$ for $i=2,3,\dots,n_1-1$ (resp. $Y_j$ for $j=2,3,\dots,n_2-1$) are chosen orthogonal to $\bm{1}_{(n_1-1) \times 1}$ (resp. $\bm{1}_{(n_2-1) \times 1}$). 
		Then $ \left[
		0 \quad 0 \quad X^T_i \quad \bm{0}_{1 \times (n_2-1)}
		\right]^T$ 
		$\big(\text{resp.} \left[
			0 \quad  0 \quad \bm{0}_{1 \times (n_1-1)} \quad Y^T_j
		\right]^T\big)$
       is an eigenvector of $AD(S_{n_1,n_2})$ corresponding to the eigenvalue $0$.

		This implies that for $i=2,3,\dots,n_1-1,$ and for $j=2,3,\dots,n_2-1$, the vectors
		$\left[
		0 \quad 0 \quad X_i \quad \bm{0}_{1 \times (n_2-1)}
		\right]^T$ and 
		$\left[
	 0 \quad 0 \quad \bm{0}_{1 \times (n_1-1)} \quad Y_j
	 \right]^T$ are eigenvectors of $AD(S_{n_1,n_2})$ corresponding to the eigenvalues $0$ with multiplicity $n_1-2$ and $0$ with multiplicity $n_2-2$ respectively. Therefore $0$ is an eigenvalue of $AD(S_{n_1,n_2})$ with multiplicity $n_1+n_2-4$. In this way we obtained $n_1+n_2-4$ eigenvalues. 		  
	 All these eigenvectors are orthogonal to $\left[
	 1 \quad 0 \quad \bm{0}_{1 \times (n_1-1)} \quad \bm{0}_{1 \times (n_2-1)}
	 \right]^T$,
	 $\left[
	 0 \quad 1 \quad \bm{0}_{1 \times (n_1-1)} \quad \bm{0}_{1 \times (n_2-1)}
	\right]^T$,
	  $\left[
	 0 \quad 0 \quad J_{1 \times (n_1-1)} \quad \bm{0}_{1 \times (n_2-1)}
	\right]^T$ 	 	 and 	 	 $\left[
	 	0 \quad 0 \quad \bm{0}_{1 \times (n_1-1)} \quad J_{1 \times (n_2-1)}
	\right]^T$.

	 Hence, the space spanned by these four vectors and the space spanned by the remaining four eigenvectors of $AD(S_{n_1,n_2})$ are same. Therefore, the remaining eigenvectors of $AD(S_{n_1,n_2})$ are of the form 	  $$X=\left[
	 \alpha_1 \quad \alpha_2 \quad \alpha_3J_{1 \times (n_1-1)} \quad \alpha_4 J_{1 \times (n_2-1)}
	 \right]^T$$ for some $(\alpha_1, \alpha_2, \alpha_3, \alpha_4) \neq (0,0,0,0)$.  	 
	 If $\beta$ is an eigenvalue of $AD(S_{n_1,n_2})$ corresponding to the eigenvector $X$, then \[AD(S_{n_1,n_2})X = \beta X, \]	 
%
	 which is equivalent to the system 	 
	 \[\begin{bmatrix}
	 -\beta & 1  & n_1-1 & 0\\
	 1 &  -\beta & 0 & n_2-1\\
	 1 & 0 & -\beta & 3(n_2-1)\\
	 0 & 1 & 3(n_1-1) & -\beta
	 \end{bmatrix} 
	 \begin{bmatrix}
	 \alpha_1 \\ \alpha_2 \\ \alpha_3 \\ \alpha_4
	 \end{bmatrix} = 
	 \begin{bmatrix}
	  0 \\ 0 \\ 0 \\ 0
	 \end{bmatrix}.\] 
	 
	 The above system admits a nontrivial solution if and only if the square matrix on the left-hand side is non-singular. Hence, the remaining four eigenvalues of $AD(S_{n_1,n_2})$ are the eigenvalues of this matrix
	 \[\begin{bmatrix}
	 	0 & 1 & n_1-1 & 0\\
	 	1 & 0 & 0 & n_2-1\\
	 	1 & 0 & 0 & 3(n_2-1)\\
	 	0 & 1 & 3(n_1-1) & 0
	 \end{bmatrix}\]
	 whose characteristic polynomial is $x^4-\left(9n_1n_2-8n_1-8n_2+8\right)x^2+4\left(n_1n_2-n_1-n_2+1\right)$.
	 This concludes the proof.	 
	\end{proof}
	
	\section{Determinant}\label{s5}
	\begin{defn}\normalfont
		
		Let $G$ be a connected graph. A set $P$ of vertices in $G$, with $|P|\geq 3$ is said to be an \textit{adjacency-diametrical cycle} of $G$ if the vertices in $P$ admit a cyclic ordering in which each pair of consecutive vertices are either adjacent or antipodal. 
		
		A partition $\mathcal S= \left\{P_1,P_2,\dots,P_k \right\}$ of a subset of $V(G)$ is said to be a \textit{adjacency-diametrical partition} of $G$ if its each part $P_i$ satisfies one of the following: (i) $P_i$ is of size $\geq 3$ and is an adjacency-diametrical cycle of $G$, (ii) $P_i$ is of size $2$ and its vertices are either adjacent or antipodal.
		
		In addition, if $\mathcal S$ is a partition of $V(G)$, then it is said to be a \textit{spanning adjacency-diametrical partition} of $G$.
		
		Let $p(\mathcal S)$ and $p_{1}(\mathcal S)$ denote the number of parts in an adjacency-diametrical partition $\mathcal S$ of $G$ which have size $2$ and $\geq 3$, respectively. 
	\end{defn}

	\begin{example}\normalfont
		For the graph $G$ shown in Figure~\ref{graph2},		
	 $\{\{1,7\},\{2,3,5,6\}\}$ is an adjacency-diametrical partition of $G$, and $\{\{1,3,5\},\{2,6\},\{7,4\}\}$ is a spanning adjacency-diametrical partition of $G$.	
	\end{example}
	 \begin{figure}[ht]
		\begin{center}
			\includegraphics[scale=1.0]{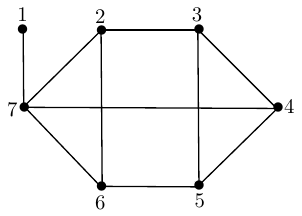}\caption{The graph $G$}\label{graph2}
		\end{center}
	\end{figure}
	\begin{lemma}\label{one-one-cycle}
		Let $G$ be a connected graph with at least three vertices. There is a one-to-one correspondence between cycles of length $k$ in the associated weighted graph $\mathcal G$  and adjacency–diametrical cycles of length $k$ in $G$.
	\end{lemma}
	
	\begin{proof}
		Let the diameter of $G$ be $d$.  Given a cycle $C_{\mathcal G}$ in $\mathcal G$, traverse it in cyclic order. Every edge of $C_{\mathcal G}$ is either a unit-weight edge that corresponds to an adjacency (an edge) of $G$, or a weight-$d$ edge that corresponds to an antipodal pair of vertices in $G$. Replacing each weight-$d$ edge of $C_{\mathcal G}$ by the corresponding pair of antipodal vertices in $G$, and interpreting unit edges as usual adjacency edges gives a cycle in $G$ that alternates between adjacencies and antipodal jumps, which is precisely an adjacency–diametrical cycle in $G$. Conversely, any adjacency–diametrical cycle of $G$ contains adjacencies of $G$ and antipodal relations; collapsing each antipodal relation to a single weight-$d$ edge yields a unique cycle of $\mathcal G$. These two constructions are inverse to each other, establishing the bijection. With the bijection established, the parity of cycle lengths is preserved by this correspondence.
	\end{proof}
	
	By an \textit{elementary subgraph of $\mathcal{G}_G$} we mean a subgraph of $\mathcal{G}_G$ (whose edge weights are $1$ or $d$),  and whose components are either a  cycle or an edge.
	
	The following result is a direct consequence of the facts discussed so for in this section.
	\begin{cor}\label{one-one2}
Let $G$ be a connected graph with diameter $d$.	There is a one-to-one correspondence between an adjacency-diametrical partition $\mathcal S = \{P_1, P_2, \dots, P_k\}$ of $G$ and an elementary subgraph of $\mathcal{G}_G$ with $k$ components (each of whose edge weights are $1$ or $d$), and whose component is a  cycle when $|P_i| \ge 3$, and an edge when $|P_i| = 2$.
	\end{cor}
	
		Now we proceed to determine the determinant of the AD matrix of a connected graph.
	\begin{thm}\label{thm 5.1}
		Let $G$ be a connected graph with vertex set $V(G)=\{1,2,\dots,n\}$ and having diameter $d$. Then
		\begin{equation}\label{e5.1}
			\det (AD(G)) = \sum (-1)^{n-p(\mathcal S)-p_1(\mathcal S)} 2^{p_{1}(\mathcal S)} d^{2 a(\mathcal S)+a_{1}(\mathcal S)},
		\end{equation}
		where the summation is over all the spanning adjacency-diametrical partitions $\mathcal S$ of $G$. Here $a(\mathcal S)$ and $a_{1}(\mathcal S)$ represents the number of pairs of antipodal vertices in all the parts $P_i$ of size $2$ and $\ge 3$ in $\mathcal S$, respectively.
	\end{thm}
	\begin{proof}
		We consider the associated weighted graph $\mathcal{G}_G$ of $G$. Let $A(\mathcal{G}_G)=(d_{ij})$. We have
		\begin{equation}\label{e5.2}
			\det (A(\mathcal{G}_G))=\overset{}{\underset{\pi}{\sum}} sgn(\pi) d_{1\pi(1)}\dots d_{n\pi(n)},
		\end{equation}
		where the summation is overall the permutations $\pi$ on $\{1, 2\dots,n\}$. Consider a nonzero term
		\begin{center}
			$d_{1\pi(1)}\dots d_{n\pi(n)}$.
		\end{center}
		Since $\pi$ admits a cycle decomposition, such a term will correspond to some 2 cycle $(ij)$ of $\pi$ which designate either an edge joining $i$ and $j$ with weight 1 or $d$, as well as some cycles of higher order, which correspond to cycles of $\mathcal{G}_G$. (Note that $\pi(i)\neq i$ for any $i$.) Thus, each nonzero term in the summation arises from an elementary weighted subgraph of $\mathcal{G}_G$. Suppose the term $d_{1\pi(1)}\dots d_{n\pi(n)}$ corresponds to the elementary subgraph $\mathcal{H}$ of $\mathcal{G}_G$. Then the sign of $\pi$ is given by,
		\begin{center}
			$sgn(\pi)=(-1)^{n-\text{number of cycles in }\pi}$,
		\end{center} 
		which equals $(-1)^{n-c(\mathcal{H})-c_1(\mathcal{H})}$, where $c(\mathcal{H})$ and $c_1(\mathcal{H})$ denote the number of edge components and number of cycle components in $\mathcal{H}$, respectively. Each spanning elementary subgraph $\mathcal{H}$  contributes $2^{c_{1}(\mathcal{H})}$ terms in the summation, since each  cycle of $\mathcal{G}_G$ can be associated to a cyclic permutation in two distinct ways. In $\mathcal{H}$, some edge component may have weight $d$, which contributes $d^{2t(\mathcal{H})}$ to the summation in~\eqref{e5.2}, where $t(\mathcal{H})$ denote the number of edge components of $\mathcal{H}$ having weight $d$. Also, some cycle components of $\mathcal{H}$ may contain edges with weight $d$, which contributes $d^{t_{1}(\mathcal{H})}$ to the summation in~\eqref{e5.2}, where $t_{1}(\mathcal{H})$ denote the number of edges having weight $d$ in all the  cycle components of $\mathcal{H}$. Hence, combining these observations, \eqref{e5.2} becomes
		
		\begin{equation}\label{H1}
			\det (A(\mathcal{G}_G)) = \sum (-1)^{n-c(\mathcal{H})-c_1(\mathcal{H})} 2^{c_{1}(\mathcal{H})} d^{2 t(\mathcal{H})+t_1(\mathcal{H})},
		\end{equation}
		where the summation is over all the spanning elementary  subgraphs $\mathcal{H}$ of $\mathcal{G}_G$. In view of Corollary~\ref{one-one2},  each spanning elementary  subgraph $\mathcal{H}$ of $\mathcal{G}_G$ corresponds to a spanning adjacency-diametrical partitions $\mathcal S$ of $G$. Moreover, it is evident that $c(\mathcal{H})=p(\mathcal S)$, $c_1(\mathcal{H})=p_1(\mathcal S)$, $t(\mathcal{H})=a(\mathcal S)$, and $t_1(\mathcal{H})=a_1(\mathcal S)$. Since $AD(G) =A(\mathcal{G}_G)$, applying these information in~\eqref{H1}, we get the desired result.
		\end{proof}

	\begin{example}\normalfont
		 For the star graph $S_4$ with central vertex 1 and pendent vertices  2, 3, 4, the spanning adjacency-diametrical partitions of    are $\mathcal S_1=\{\{1, 2, 3, 4\}\}$, $\mathcal S_2=\{\{1,2,4,3\}\}$, $\mathcal S_3=\{\{1,3,2,4\}\}$, $ \mathcal S_4=\{\{1,2\}, \{3,4\}\}$, $\mathcal S_5=\{\{1,3\},\{2,4\}\}$, and $\mathcal S_6=\{\{1,4\},\{2,3\}\}$.		
		Note that		
		$p(\mathcal S_i)=a(\mathcal S_i)=p_1(\mathcal S_j)=a_1(\mathcal S_j)=0$ for $i=1,2,3$ and $j=4,5,6$;		
		$p(\mathcal S_i)=p_1(\mathcal S_j)=a(\mathcal S_j)=1$ for $i=1,2,3$ and $j=4,5,6$;		
		$a_1(\mathcal S_i)=2$ for $i=1,2,3$.		
		Applying these values in~\eqref{e5.1}, we get
		\[\det(AD(S_4))= 3\,(-1)^3\,2^1\,2^2 + 3\,(-1)^2\,2^0\,2^{2(1)}=-12.\]
		
		\noindent This coincide with the direct computation of the determinant of $AD(G)$.
	\end{example} 

	\begin{example}\normalfont
		Consider the path graph $P_n$ on $n$ $(\geq 2)$ vertices. We show that
		\[
		\det(AD(P_n))=
		\begin{cases}
			2(n-1), & \text{if}~ n \text{ is odd};\\
			-1, & \text{if}~ n=2;\\
			(n-2)^2, & \text{if}~n=4k \text{ for } k=1,2,\dots;\\
			-n^2, & \text{if}~n=4k+2 \text{ for } k=1,2,\dots \end{cases} 
		\]		
		
		Clearly the diameter of $P_n$ is $n-1$. We assume that $V(P_n)=\{1,2,\ldots, n\}$ and $E(P_n)=\{\{i, i+1\} : i=1,2,\ldots, n-1\}$. We proceed by considering three cases.
		
		\noindent \textit{Case 1.}
		Suppose $n$ is odd.
		
		In this case $P_n$ has $\mathcal S=\{\{1,2,\dots,n\}\}$ as the unique spanning adjacency-diametrical partition, and  $p(\mathcal S)=a(\mathcal S)=0$ and  $p_1(\mathcal S)=a_1(\mathcal S)=1$, Applying these in~\eqref{e5.1}, we get
		\begin{center}
			$\det(AD(P_n))= (-1)^{n-1}\,2^1\,(n-1)^1=2(n-1)$.
		\end{center}

		\noindent \textit{Case 2.}
		Suppose $n$ is even. 
		
	Let $n=2$. The only spanning adjacency-diametrical partition of $P_2$ is $\mathcal S=\{\{1,2\}\}$, and $p(\mathcal S)=a(\mathcal S)=1$ and $p_1(\mathcal S)=a_1(\mathcal S)=0$. Then from~\eqref{e5.1}, we get
		\begin{center}
			$\det(AD(P_2))= (-1)^{2-1}\,2^0\,1^2=-1$.
		\end{center} 
	
		Let $n=4k$ or $n=4k+2$, where $k=1,2,\dots$. In either case,  the spanning adjacency-diametrical partitions for $P_n$ are the following:		
		$\mathcal S_1=\{\{1,2,\dots,n\}\}$, $\mathcal S_2=\{\{1,2\}, \{3,4\}, \dots, \{n-1, n\}\}$ and $\mathcal S_3=\{\{1,n\}, \{2,3\},\dots,\{n-2,n-1\}\}$. Here $p(\mathcal S_1)=a(\mathcal S_1)=p_1(\mathcal S_2)=a(\mathcal S_2)=a_1(\mathcal S_2)=p_1(\mathcal S_3)=a(\mathcal S_3)=0$, $p_1(\mathcal S_1)=a_1(\mathcal S_1)=a_1(\mathcal S_3)=1$ and $p(\mathcal S_2)=p(\mathcal S_3)=\frac{n}{2}$.
				
		When $n=4k$ for $k=1,2,\dots$, applying these in~\eqref{e5.1}, we obtain
		\begin{center}
			$\det(AD(P_n))= (-1)^{n-1}\,2^1\,(n-1)^1 + (-1)^{n-\frac{n}{2}}+(-1)^{n-\frac{n}{2}}\,(n-1)^2 =(n-2)^2$.
		\end{center}        
	
		Similarly, when $n=4k+2$ for $k=1,2,\dots$, from~\eqref{e5.1}, we get
	     \begin{center}
	     	$\det(AD(P_n))= (-1)^{n-1}\,2^1\,(n-1)^1 + (-1)^{n-\frac{n}{2}}+(-1)^{n-\frac{n}{2}}\,(n-1)^2 =-n^2$.
	     \end{center}                   
	\end{example}
	
		\begin{thm}\label{thm 5.3}
		Let $G$ be a connected graph with $n$ vertices and  diameter $d$. Let
		\begin{equation*}
			\Phi_{AD(G)}(\lambda)=\det(\lambda I-AD(G))=\lambda^{n}+c_{1}\lambda^{n-1}+\dots+c_n
		\end{equation*}
		be the characteristic polynomial of $AD(G)$. Then
		\begin{equation}\label{e5.3}
			c_k=\sum (-1)^{p(\mathcal S)+p_1(\mathcal S)} 2^{p_{1}(\mathcal S)} d^{2 a(\mathcal S)+a_1(\mathcal S)},
		\end{equation} 
		where the summation is over all the adjacency-diametrical partitions $\mathcal S$ of $G$ with $k$ vertices, for each $k=1,2,\dots,n$. 
	\end{thm}
	\begin{proof}
		Note that the number $(-1)^k c_k$ is the sum of all $k \times k$ principal minors of $AD(G)$, $k=1,2,\dots,n$. By Theorem~\ref{thm 5.1},
		\begin{equation*}
			c_k=(-1)^{k}\sum (-1)^{k-p(\mathcal S)-p_1(\mathcal S)} 2^{p_{1}(\mathcal S)} d^{2 a(\mathcal S)+ a_1(\mathcal S)},
		\end{equation*}
		where the summation is over all the adjacency-diametrical partitions $\mathcal S$ of $G$ with $k$ vertices, for each $k=1,2,\dots,n$. Hence, $c_k=\sum (-1)^{p(\mathcal S)+p_1(\mathcal S)} 2^{p_{1}(\mathcal S)} d^{2 a(\mathcal S)+a_1(\mathcal S)}$. Note that $c_{1}=0$.
	\end{proof}
	
	\begin{note}\normalfont
		We now obtain the explicit expressions of $c_2$ and $c_3$ by applying Theorem~\ref{thm 5.3}.

		(i) Adjacency-diametrical partitions  of $G$ with two vertices is of the form either $\mathcal S_1=\{\{v_i,v_j\}\}$, where $v_i \sim v_j$ or $\mathcal S_2=\{\{v_i,v_j\}\}$, where $d_G(v_i,v_j)=d$.
		
		For both types we have $p(\mathcal S_1)=p(\mathcal S_2)=a(\mathcal S_2)=1$, $p_1(\mathcal S_1)=p_1(\mathcal S_2)=a(\mathcal S_1)=a_1(\mathcal S_1)=a_1(\mathcal S_2)=0$. Substituting these values in~\eqref{e5.3}, we get
		
		\[c_2={\sum} (-1)^{1}\,2^{0}\, d^{0}+{\sum} (-1)^{1} \,2^{0}\, d^{2(1)},\]
		where the first and second sums range over all adjacency–diametrical partitions of the forms $\mathcal S_1$ and $\mathcal S_2$, respectively. Hence,		
		$$c_2=-(|E(G)|+kd^2),$$
		where $k$ denotes the number of unordered antipodal pairs of vertices in $G$.
		
		(ii) Adjacency-diametrical partitions  of $G$ with three vertices fall into the following four types: $\mathcal S_1=\{\{v_i,v_j,v_k\}\}$, where $v_i \sim v_j$, $v_j \sim v_k$ and $v_i \sim v_k$; $\mathcal S_2=\{\{v_i,v_j,v_k\}\}$, where $v_i \sim v_j$, $v_j \sim v_k$ and $d_G(v_i,v_k)=d$; $\mathcal S_3=\{\{v_i,v_j,v_k\}\}$, where $v_i \sim v_j$, $d_G(v_j,v_k)=d$ and $d_G(v_i,v_k)=d$; $\mathcal S_4=\{\{v_i,v_j,v_k\}\}$, where $d_G(v_i,v_j)=d$, $d_G(v_j,v_k)=d$ and $d_G(v_i,v_k)=d$.
		
	For $i=1,2,3$, we have $p(\mathcal S_i)=a(\mathcal S_i)=a_1(\mathcal S_1)=0$, $p_1(\mathcal S_i)=a_1(\mathcal S_2)=1$ with $a_1(\mathcal S_3)=2$ and $a_1(\mathcal S_4)=3$. Substituting these in~\eqref{e5.3}, gives
		\[
		c_3=
			\sum (-1)^{1}\,2^{1}\,d^{0}
			+\sum (-1)^{1}\,2^{1}\,d^{1}
			+\sum (-1)^{1}\,2^{1}\,d^{2}
			+\sum (-1)^{1}\,2^{1}\,d^{3}\\
		\]
			where the four sums correspond to  adjacency-diametrical partitions  of $G$ of the form $\mathcal S_1$, $\mathcal S_2$, $\mathcal S_3$, and $\mathcal S_4$, respectively. Therefore,
		$$c_3=	 -2(k_0+k_1 d+k_2 d^{2}+k_3 d^{3}),$$			
		where for $s\in\{0,1,2,3\}$, $k_s$ is the number of number of unordered triples $\{v_i,v_j,v_k\}$ that have exactly $s$ antipodal pairs and the other 
		$3-s$ pairs at distance 1. Note that $k_0$ is the number of triangles in $G$.
	\end{note}

	Note that many results on the spectra of adjacency matrices of unweighted graphs in spectral graph theory extend directly to weighted graphs without any modification. In the forthcoming results, we will use these facts without explicitly stating this extension each time.

	We obtain the next result as a direct consequence of \cite[Corollary 3.11]{bapat2010graphs}, using the fact that $AD(G) =A(\mathcal{G}_G)$, and that adjacency-diametrical cycles of $G$ can be viewed as cycles in $\mathcal{G}_G$ (cf. Lemma \ref{one-one-cycle}). 	
	\begin{cor}\label{cor 5.1}
		Let $G$ be a connected graph on $n$ vertices. Let
		\begin{equation*}
			\Phi_{AD(G)}(\lambda)=\det(\lambda I-AD(G))=\lambda^{n}+c_{1}\lambda^{n-1}+\dots+c_n
		\end{equation*}
		be the characteristic polynomial of $AD(G)$. Suppose $c_{3}=c_{5}=\dots=c_{2k-1}=0$. Then $G$ has no odd adjacency-diametrical cycle of length $i, 3\leq i \leq 2k-1$. Furthermore, the number of adjacency-diametrical  cycles of length $2k+1$ in $G$ is $-\frac{1}{2} c_{2k+1}$. 
	\end{cor}
	
\subsection{Diametrical bipartite graphs}
	Now we introduce a  special class of bipartite graphs.
	\begin{defn}\normalfont
		A connected graph $G$ is said to be \textit{diametrical bipartite} if it is bipartite with bipartition $(V_1, V_2)$ and no two vertices in the same part are antipodal.
	\end{defn}
	
	\begin{lemma}\label{odddiam}
		A connected bipartite graph is diametrical bipartite if and only if its diameter $d$ is odd.
	\end{lemma}
	
	\begin{proof}
		Let $G$ be a connected bipartite graph with bipartition $(V_1, V_2)$ and diameter $d$.
		Since $G$ is bipartite, every pair of vertices lying in the same part $V_i$ has even distance, because any path in a bipartite graph alternates between the two parts.
		
		If $d$ is odd, then no two vertices in the same part can be at distance $d$, and therefore $G$ is diametrical bipartite.
		
		If $d$ is even, then any pair of vertices $x,y$ with $d_G(x,y)=d$ must lie in the same bipartition class (as they are at even distance). Hence some two vertices in the same part are antipodal, so $G$ is not diametrical bipartite.
	\end{proof}

	\begin{thm}\label{thm 5.2}\label{diam bipartite}
		Let $G$ be a connected graph with at least three vertices. Then the following are equivalent.		
		\begin{enumerate}[\normalfont (i)]
			\item $G$ is diametrical bipartite
			\item $\mathcal{G}_G$ is bipartite
			\item All the adjacency-diametrical cycles in $G$ are of even length
		\end{enumerate}  
		
	\end{thm}
	
	\begin{proof}
		
		Let the diameter of $G$ be $d$.	The weighted graph $\mathcal{G}_G$ is obtained from $G$ by taking a copy of $G$ and assigning unit weight to their edges, and adding  joining each pairs of antipodal vertices in $G$ by a new edge having weight $d$. 
		
		(i) $\Leftrightarrow$ (ii): Suppose $G$ is diametrical bipartite. By Lemma~\ref{odddiam}, $d$ is odd.  Consequently, any pair of antipodal vertices cannot both occur in the same part of the bipartition of $G$. So end vertices of the newly added edge with weight $d$ must lie in two different parts. Thus $\mathcal{G}_G$ is bipartite. By retracing the above argument, we get the proof of converse part.
		
		(ii) $\Leftrightarrow$ (iii): In view of Lemma~\ref{one-one-cycle}, 
		$\mathcal G$ is bipartite (every cycle in $\mathcal G$ has even length) if and only if every adjacency–diametrical cycle in $G$ has even length.
	\end{proof}

Since $AD(G) =A(\mathcal{G}_G)$, the next result follows directly from Theorem \ref{diam bipartite} together with \cite[Theorem 3.14]{bapat2010graphs}.
	\begin{thm}\label{thm 5.4}
		Let $G$ be a connected graph on  $n$ vertices. Then the following conditions are equivalent.
		\begin{enumerate}[\normalfont (i)]
			\item $G$ is diametrical bipartite; 
			\item If $ \Phi_{AD(G)}(\lambda)= det(AD(G)-\lambda I)=\lambda^n + c_1\lambda^{n-1}+ \dots +c_n$ is the characteristic polynomial of $AD(G)$, then
			$c_{2k+1} = 0, k = 0,1,\dots$;
			\item The eigenvalues of $AD(G)$ are symmetric with respect to the origin.
		\end{enumerate}
	\end{thm}
	
	\section{Some graph invariants and bounds}\label{s6}
	\begin{defn}\normalfont
		Let $G$ be a connected graph with diameter $d$. The \textit{diametrical degree} of a vertex $u$ in $G$, denoted by $\hat{d_G}(u)$ or simply by $\hat{d}(u)$, is the number of vertices $v$ in $G$ such that $d_G(u,v)=d$, i.e, 
		\[\hat{d}(u)=|\{v \in V(G):d_G(u,v)=d\}|.\]		
		We denote by $\hat{\Delta}(G)$ and $\hat{\delta}(G)$, the maximum and the minimum of the diametrical degrees of the vertices of $G$, respectively. 	
		The \textit{diametrical degree sum} of $G$, denoted by $\hat{d}(G)$, is the sum of all the diametrical degrees of the vertices of $G$, i.e.,
		\begin{equation*}
			\hat{d}(G)=\overset{}{\underset{v \in V(G)}{\sum}}\hat{d}(v).
		\end{equation*} 
	\end{defn}
	
	\begin{example}\normalfont
		The diametrical degrees of the vertices of the graph $G$ shown in Figure~\ref{graph1} are $\hat{d}(v_1)=0, \hat{d}(v_2)=0, \hat{d}(v_3)=1, \hat{d}(v_4)=0, \hat{d}(v_5)=0,$ and $\hat{d}(v_6)=1$. And the diametrical degree of $G$ is $\hat{d}(G)=0+0+1+0+0+1=2$. Moreover, $\hat{\delta}(G)=0$ and $\hat{\Delta}(G)=1$.
			\begin{figure}[H]
					\begin{center}
							\includegraphics[scale=1.0]{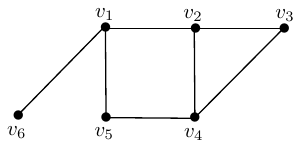}\caption{The graph $G$}\label{graph1}
						\end{center}
				\end{figure}		
	\end{example}
	
	\begin{lemma}\label{thm 6.1}
		Let $G$ be a connected graph with $n$ vertices, $m$ edges and diameter $d$. Let $\lambda_1 \ge \lambda_2 \ge \dots \ge \lambda_n$ be the eigenvalues of $AD(G)$. Then 
		\begin{enumerate}[\normalfont (i)]
			\item $\overset{n}{\underset{i=1}{\sum}}\lambda_{i}=0$;
			\item $\overset{n}{\underset{i=1}{\sum}} \lambda^2_{i}=2m+d^{2}\hat{d}(G)$.
		\end{enumerate}
	\end{lemma}
	\begin{proof}
		\begin{enumerate}[(i)]
			\item Since the trace of $AD(G)$ is $0$, statement (i) follows.
			\item Let $AD(G)=(d_{ij})$. Then 
			
			$\overset{n}{\underset{i=1}{\sum}} \lambda^2_{i} = tr(AD(G))^2
			=\overset{n}{\underset{i=1}{\sum}}\overset{n}{\underset{j=1}{\sum}}d_{ij}d_{ji}
			=\overset{n}{\underset{i=1}{\sum}}\left(d(v_{i})+d^2\hat{d}(v_i)\right)
			=2m+d^2\hat{d}(G) $
		\end{enumerate}
		and the proof is complete.
	\end{proof}

The largest eigenvalue of $AD(G)$ is said to be the \textit{adjacency-diametrical spectral radius} (or \textit{AD spectral radius}) of $G$.
	
	\begin{thm}
		Let $G$ be a connected graph with $n$ vertices, $m$ edges and diameter $d$. Let $\lambda_1 \ge \lambda_2 \ge \dots \ge \lambda_n$ be the eigenvalues of $AD(G)$. Then 
		\begin{equation*}
			\lambda_1 \le \sqrt{ \left(\frac{n-1}{n}\right) (2m+d^{2}\hat{d}(G))}.
		\end{equation*}
	\end{thm}
	\begin{proof}
		From Lemma~\ref{thm 6.1}(i),we have $\overset{n}{\underset{i=1}{\sum}}\lambda_{i}=0$. 
		Therefore, $\lambda_1=- \overset{n}{\underset{i=2 }{\sum}}\lambda_{i}$ and hence
		\begin{equation}\label{e6.1}
			\lambda_1\leq \sum_{i=2}^{n} |\lambda_{i}|.
		\end{equation}
	Using Cauchy-Schwarz inequality, we get
		\begin{equation}\label{e6.2}
			\left(\overset{n}{\underset{i=2 }{\sum}}|\lambda_{i}|\right)^2 \leq (n-1) \overset{n}{\underset{i=2 }{\sum}}\lambda^2_{i}.
		\end{equation} 
		By Lemma~\ref{thm 6.1}(ii) and using~\eqref{e6.1},~\eqref{e6.2}, we get
		\[
		2m+d^2\hat{d}(G)-\lambda^{2}_1 =\sum_{i=2}^{n} \lambda^2_{i}
		\geq\frac{1}{n-1}\left(\sum_{i=2}^{n} |\lambda_{i}|\right)^2 
		\geq\frac{\lambda^{2}_1}{n-1}.\]
		Hence 
		\begin{eqnarray*}
			2m+d^2\hat{d}(G)\geq \lambda^2_{1}\left(1+\frac{1}{n-1}\right)=\lambda^2_{1}\left(\frac{n}{n-1}\right),
		\end{eqnarray*}
		which leads to the required result.
	\end{proof}
	
	\begin{thm}
		Let $G$ be a adjacency-diametrical bipartite graph with $n$ vertices, $m$ edges and diameter $d$. Let $\lambda_1 \geq \lambda_2 \geq \dots \geq \lambda_n$ be the eigenvalues of $AD(G)$. Then 
		\[\lambda_n \leq \sqrt{m+\frac{1}{2}d^2\hat{d}(G)}.\]
	\end{thm}
	\begin{proof}
	From Theorem~\ref{thm 5.4}(iii), we know that $\lambda_1=-\lambda_n$ and Lemma~\ref{thm 6.1} gives $\overset{n}{\underset{i=1}{\sum}}\lambda^2_{i}=2m+d^{2}\hat{d}(G)$.
		Hence
			$\lambda_1^{2}+\lambda_n^{2} \leq 2m+d^{2}\hat{d}(G)$,
	which yields		$2\lambda_n^{2}\leq 2m+d^{2}\hat{d}(G)$.	
		This completes the proof.
	\end{proof}

	An independent set $S$ of $G$ is said to be an  \textit{adjacency-diametrical independent set} (or \textit{AD independent set}) of $G$, if no two vertices of $S$ are antipodal.
 	The maximum cardinality of a adjacency-diametrical independent set in $G$, denoted by $\alpha_{AD}(G)$,  is said to be the \textit{adjacency-diametrical independence number} (or \textit{AD independence number}) of $G$~\cite{rajkumar2025-1}.	
	Clearly $\alpha_{AD}(G) \leq \alpha(G).$

		For the graph $G$ shown in Figure~\ref{graph5}, the subset $\{v_3,v_4,v_5\}$ is a AD independent set of $G$ and it is the AD independent set of $G$ of maximum cardinality. So the AD independence number $\alpha_{AD}(G)=3$.
		\begin{figure}[H]
			\begin{center}
				\includegraphics[scale=1.0]{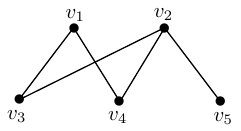}\caption {The graph $G$} \label{graph5}
			\end{center}
		\end{figure}
	Not every independent set is an AD independent set. For instance, in the graph shown in Figure~\ref{graph5}, the subset $\{v_1,v_5\}$ is an independent set, but not an AD independent set, since $v_1$ and $v_5$ are antipodal vertices.

		It is evident that the AD independence number of $G$ equals the independence number of $\mathcal{G}_G$.
	Since $AD(G) =A(\mathcal{G}_G)$, applying~\cite[Theorem 3.10.1]{cvetkovic2010introduction} yields the following result.
	\begin{thm}\label{thm 6.4}
		Let $G$ be a connected graph on $n$ vertices. Let $n^+$ and $n^-$ denote the number of positive and negative eigenvalues of $AD(G)$ respectively. Then
		\begin{equation*}
			\alpha_{AD}(G)\leq \min \{n-n^+,n-n^-\}.
		\end{equation*}
	\end{thm}

	    A proper coloring of a connected graph $G$ is said to be an \textit{adjacency-diametrical coloring} (or \textit{AD-coloring}) of $G$, if no two antipodal vertices have same color.
		The \textit{adjacency-diametrical chromatic number} (or  \textit{AD chromatic number}) of $G$, denoted by $\chi_{AD}(G)$, is the minimum number of colors needed for an adjacency-diametrical coloring of $G$. For each color, the set of all vertices which get that color is an AD independent and is called an \textit{adjacency-diametrical color class}. Equivalently, $\chi_{AD}(G)$ is the cardinality of minimum partition of $V(G)$ into AD independent subsets~\cite{rajkumar2025-2}.

	Note that for any connected graph $G$, $\chi(G) \leq \chi_{AD}(G)$.
	For example, the path graph $P_5$ has $\chi(P_5)=2$ and  $\chi_{AD}(P_5)=3$.

 It is clear that the AD chromatic number of $G$ equals the chromatic number of $\mathcal{G}_G$. 	Since $AD(G)=A(\mathcal{G}_G)$, applying \cite[Theorem~3.23]{bapat2010graphs} directly implies the following result.

	\begin{thm}\label{thm 6.5}
		Let $G$ be a connected graph with $n$ vertices and with at least one edge. Let $\lambda_1 \geq \lambda_2 \geq \dots \geq \lambda_n$ be the eigenvalues of $AD(G)$. Then
		\begin{equation*}
			\chi_{AD}(G) \geq 1 - \frac{\lambda_1}{\lambda_n}.
		\end{equation*}	
	\end{thm}

	\begin{cor}\label{cor6.1}
		Let $G$ be a connected, planar graph on $n$ vertices. Let $\lambda_1 \geq \lambda_2 \geq \dots \geq \lambda_n$ be the eigenvalues of $AD(G)$. Suppose $\chi(G)= \chi_{AD}(G)$. Then $\lambda_1 \leq -3\lambda_n$.
	\end{cor}
	\begin{proof}
		It is known that every planar graph is $4-$colorable~\cite{bondy1976}. It follows that $\chi_{AD}(G)\leq 4$. Consequently, Theorem~\ref{thm 6.5} yields $\lambda_1 \leq -3\lambda_n$.
	\end{proof}

	There are several graphs for which the condition $\chi(G)= \chi_{AD}(G)$ of Corollary~\ref{cor6.1} holds. Examples include  the path $P_{2n}$, the cycle $C_n$ for $n=4k+2$ with $k=1,2,\dots$, and double star.

	\begin{defn}\normalfont
	Let $G$ be  a connected graph with diameter $d$. The \textit{adjacency-diametrical degree} (or \textit{AD degree}) of a vertex $v$ in $G$, denoted by $\doublehat{d}_{G}(v)$ (or simply $\doublehat{d}(v)$), is defined as
		\begin{equation*}
			\doublehat{d}_G(v)=d(v)+d \hat{d}(v).
		\end{equation*}
		The maximum and the minimum of the AD degrees of vertices in $G$ are denoted by $\doublehat{\Delta}(G)$ and $\doublehat{\delta}(G)$, respectively.
		If every vertex of $G$ has the same AD degree $r$, then $G$ is called the \textit{adjacency-diametrical regular} (or \textit{AD regular}) of degree $r$ or simply $r$-\textit{AD regular}.
		
		The \textit{mean adjacency-diametrical degree} of $G$, denoted by $\bar{d}_{AD}(G)$,  is defined as \[\bar{d}_{AD}(G)=\frac{1}{n} \overset{}{\underset{v \in V(G)}{\sum}} \doublehat{d}_{G}(v).\] 
	\end{defn}

Observe that the AD degree of a vertex in $G$ coincide with the weighted degree of the corresponding vertex in $\mathcal{G}_G$. Consequently, the maximum and minimum AD degrees of $G$ agree with the maximum and minimum weighted degrees of $\mathcal{G}_G$ respectively. Hence, $G$ is AD regular if and only if $\mathcal{G}_G$ is weighted regular. Moreover, the mean adjacency-diametrical degree of $G$ is equal to the mean weighted degree of $\mathcal{G}_G$.
	
	\begin{example}\normalfont
		Consider the cycle graph $C_6$. Each vertex in $C_6$ has degree $2$ and diametrical degree $1$,  and the diameter of $C_6$ is $3$. Hence the AD degree of every vertex  in $C_6$ is $5$. Thus $C_6$ is $5$-AD regular.	
	\end{example}

	\begin{rem}\normalfont
	 Every distance-regular graph is AD-regular. However, not every AD regular graph is distance-regular. For instance, the hexagonal prism graph shown in Figure~\ref{F4a} is AD regular but not distance regular. Also, not every regular graph is AD-regular.  For instance, the Frucht graph shown in Figure~\ref{F4b} is regular but not AD regular. Moreover, we have not discussed whether every AD regular graph must also be regular, nor have we provided an example showing that AD regular does not imply regular.
	 
	 \begin{figure}[h!]	
	 	\centering 
	 			  \begin{subfigure}{0.4\linewidth}
	 			  	\centering 		\includegraphics[scale=1]{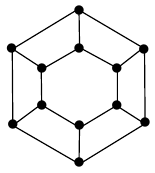}
	 		\caption{}\label{F4a}
	 	\end{subfigure} \hspace{.3cm}
	 	\begin{subfigure}{0.3\linewidth}
	 		\centering 
	 		\includegraphics[scale=.9]{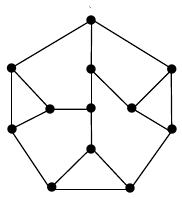}
	 		\caption{}\label{F4b}
	 	\end{subfigure} 
	 	\caption{(a) The hexagonal prism graph, (b) The Frucht graph.}\label{F4}
	 \end{figure}
	\end{rem}
	
	Notice that the AD degree of each vertex in $G$ equals the weighted degree of the corresponding vertex in $\mathcal{G}_G$. Then from~\cite[Proposition 1.1.2]{cvetkovic2010introduction} and~\cite[Theorem 3.2.1]{cvetkovic2010introduction}, we obtain the following two results.

	\begin{pro}\label{pro 6.1}
		A connected graph $G$ is $r$-AD regular if and only if all one vector is an eigenvector of $G$ with corresponding eigenvalue $r$.
	\end{pro}

	\begin{thm}\label{thm 6.7}
		Let $G$ be a connected graph with $n$ vertices and diameter $d$. Let $\lambda_1  \ge \lambda_2 \ge \dots \ge \lambda_n$ be the eigenvalues of $AD(G)$. Then\\
		\begin{equation*}
			\doublehat{\delta} \leq \bar{d}_{AD}(G) \leq \lambda_1 \leq \doublehat{\Delta}.
		\end{equation*}
		Equality holds if and only if $G$ is AD regular.
	\end{thm}

	Notice that the mean adjacency-diametrical degree of $G$ equals the mean weighted degree of $\mathcal{G}_G$. Since $AD(G)=A(\mathcal{G}_G)$, by~\cite[Theorem 3.5.7]{cvetkovic2010introduction} we obtain the following result.

	\begin{thm}\label{thm 6.8}
	Let $G$ be a AD regular graph and $(r=)\lambda_1 \geq \lambda_2 \geq \dots \geq \lambda_n$ be the eigenvalues of $AD(G)$. Let $H$ be an induced subgraph of $G$ with $n_1$ vertices and mean adjacency-diametrical degree $\bar{d}_{AD}(H)$. Then
		\begin{equation}\label{e6.7}
			\frac{n_1(r-\lambda_n)}{n}+\lambda_n \leq \bar{d}_{AD}(H).
		\end{equation}
	\end{thm}

	Clearly $G$ is AD regular if and only if $\mathcal{G}_G$ is regular. Since $AD(G)=A(\mathcal{G}_G)$, from~\cite[Theorem 3.10.2]{cvetkovic2010introduction},
	the following result can be deduced. 
	\begin{thm}\label{thm 6.9}
		If $G$ is a AD regular graph and $(r=)\lambda_1 \geq \lambda_2 \geq \dots \geq \lambda_n$ be the eigenvalues of $AD(G)$. Then
		\begin{equation*}
			\alpha_{AD}(G) \leq n \frac{-\lambda_n}{\lambda_1 - \lambda_n}.
		\end{equation*}
	\end{thm}

Since $AD(G)=A(\mathcal{G}_G)$, applying~\cite[Theorem 3.10.3]{cvetkovic2010introduction}
,  we get the following result.
	
	\begin{thm}
		Let $G$ be a connected graph.  Let $\lambda_1 \geq \lambda_2 \geq \dots \geq \lambda_n$ be the eigenvalues of $AD(G)$. Let $t^{-}, t^{0}, t^{+}$ denote the number of eigenvalues $\lambda_{i}$ less than, equal to, or greater than $-1$ respectively. Let $m=\min \{t^0+t^{-}+1, t^0+t^{+}, 1+\lambda_{1}\}$. Then  $\omega(G) \leq m$. If $m=t^0+t^{-}+1$ and the eigenvalues greater than $-1$ is exceed $t^0+t^{-}$, then $\omega(G) \leq m-1$.
	\end{thm}

	\section{Graph products and AD-spectra}\label{s7}
	
	The join of two vertex-disjoint graphs $G$ and $H$, denoted $G \vee H$, 
	is the graph obtained from the union $G \cup H$ by adding all possible edges 
	between every vertex of $G$ and every vertex of $H$. Since the diameter of  $G \vee H$ is $2$, so the AD-spectrum of $G \vee H$ coincides the distance spectrum of $G \vee H$, which was described in~\cite[Theorem 2]{stevanovic2009thedistance} for the case when $G$ and $H$ are regular.

 The lexicographic product of two graphs $G_1 = (V_1, E_1)$ and $G_2 = (V_2, E_2)$, denoted $G_1 [G_2]$, is the graph with vertex set $V_1 \times V_2$. 
Two vertices $(u_i, v_j)$ and $(u_k, v_\ell)$ are adjacent in $G_1[G_2]$ if $u_i u_k \in E_1$ or ($u_i = u_k$ and $v_j v_\ell \in E_2)$.
	
	We now compute the AD-spectrum of the lexicographic product of two graphs.
	
	\begin{thm}
		Let $G$ be a connected graph with $n$ vertices and let $H$ be an $r$-regular graph with $m$ vertices. Let $\lambda_{1}, \lambda_2, \dots, \lambda_m$ be the AD-spectrum of $G$ and $(r=)\mu_1, \mu_2, \dots , \mu_n$ be the spectrum of $A(H)$. 
		\begin{enumerate}[(i)]
			\item If the diameter of $G[H]$ is $2$, then the AD-spectrum of $G[H]$ contains the eigenvalue $m \mu_i+2m-r-2$ with multiplicity $1$ and $-(\lambda_j+2)$ with multiplicity $n$, for $i=1,2,\dots,n$; $j=2,3,\dots,m-1$.

			\item If the diameter of $G[H]$ is greater than $2$, then the AD-spectrum of $G[H]$ contains the eigenvalue $m \mu_i+r$ with multiplicity $1$ and $\lambda_j$ with multiplicity $n$, for $i=1,2,\dots,n$; $j=2,3,\dots,m-1$.
		\end{enumerate}
		\end{thm}
	\begin{proof}
		(i) Since the diameter of $G[H]$ is 2, we have $AD(G[H]) = D(G[H])$. 
		The distance spectrum of $D(G[H])$ is given in~\cite[Theorem~3.3]{indulal2009distance}, and therefore the claim follows.
		
		(ii) Suppose the diameter of $G[H]$ is greater than $2$. Let $AD(G) = [d_{ij}]$. 
		With a suitable labeling of the vertices of $G[H]$, the AD matrix of $G[H]$ can be expressed as follows:
		\[\begin{aligned}
				AD(G[H])&=\begin{bmatrix}
				A(H) & d_{12}J_m & d_{13}J_m & \dots &  d_{1n}J_m\\
				d_{21}J_m & A(H) & d_{23}J_m & \dots &  d_{2n}J_m\\
				d_{31}J_m & d_{32}J_m & A(H) & \dots &  d_{3n}J_m\\
				\vdots & \vdots & \vdots & \ddots  & \vdots \\
				d_{(n-1)1}J_m & d_{(n-1)2}J_m & d_{(n-1)3}J_m & \dots & d_{(n-1)n}J_m\\
				d_{n1}J_m & d_{n2}J_m & d_{n3}J_m & \dots & A(H)
			\end{bmatrix}\\
			&=(AD(G) \otimes J_m) + (I_n \otimes A(H)).
		\end{aligned}
		\]
				
		Since $H$ is $r$-regular, the all one vector $\bm{1}_{m \times 1}$ is an eigenvector of $A(H)$ corresponding to the eigenvalue $r$. The remaining eigenvalues $\lambda_j$ of $A(H)$, with eigenvectors $X_j$, are orthogonal to $\bm{1}$. Let $Y_k$, $k = 1,2,\dots,n$, denote the $n$ linearly independent eigenvectors of $I_n$ associated with the eigenvalue $1$. Then, 
			 	\[\begin{aligned}
			AD(G[H]) (Y_k \otimes X_j)& =\left((AD(G) \otimes J_m) + (I_n \otimes A(H))\right) (Y_k \otimes X_j)\\
			&= (AD(G) Y_k) \otimes (J_m X_j) + (I_n Y_k) \otimes (A(H) X_j)\\
			&= 0 + (Y_k \otimes \lambda_j X_j)\\
			&= \lambda_j (Y_k \otimes X_j).
		\end{aligned}\]
		Thus,	for each $j = 2,3,\dots,m$, the vectors $Y_k \otimes X_j$ (for $k = 1,2,\dots,n$) are eigenvectors of $AD(G[H])$ corresponding to the eigenvalue $\lambda_j$. 		
		Let $Z_i$ be an eigenvector corresponding to the eigenvalue $\mu_i$ of $AD(G)$ for $i=1,2\ldots,n$. Then
		\[\begin{aligned}
				AD(G[H]) (Z_i \otimes \bm{1}_m)& =\left((AD(G) \otimes J_m) + (I_n \otimes A(H))\right) (Z_i \otimes \bm{1}_m)\\
			&= (AD(G) Z_i) \otimes (J_m \bm{1}_m ) + (I_n Z_i) \otimes (A(H)  \bm{1}_m)\\
			&= (\mu_iZ_i \otimes m\bm{1}_m)+ (Z_i \otimes r\bm{1}_m)\\
			&= (m\mu_i+r) (Z_i \otimes \bm{1}_m).
		\end{aligned}\]
		Therefore, $m\mu_i+r$ is an eigenvalue of $AD(G[H])$ with eigenvector $Z_i \otimes \bm{1}_m$. Furthermore, the $nm$ vectors of the form $Z_i \otimes \bm{1}_m$ and $Y_k \otimes X_j$ are linearly independent. Since eigenvectors associated with distinct eigenvalues are linearly independent, and $AD(G[H])$ has a basis consisting entirely of eigenvectors, the theorem is established.
		\end{proof}
		
		 Let $G_1 = (V_1, E_1)$ and $G_2 = (V_2, E_2)$ be two graphs with vertex sets $V_1 = \{u_1, u_2, \dots, u_n\}$ and $V_2 = \{v_1, v_2, \dots, v_n\}$, respectively. The Cartesian product of $G_1$ and $G_2$, denoted by $G_1 \Box G_2$, is the graph with vertex set $V_1 \times V_2$, and two vertices $(u_i, v_j)$ and $(u_k, v_\ell)$ are adjacent if	$u_i = u_k$ and $v_jv_\ell \in E_2$, or $u_iu_k \in E_1$ and $v_j = v_\ell$.
		
		Next, we determine the AD-spectrum of the Cartesian product of two graphs.
		
		\begin{thm}
			Let $G$ and $H$ be two distance regular graphs with $n$ and $m$ vertices respectively and let the diameters of $G$ and $H$ be denoted by $d_G$ and $d_H$ respectively. Let the spectrum of $A(G)$ and $A(H)$ be $\lambda_{1}, \lambda_2, \dots, \lambda_n$ and $\mu_1, \mu_2, \dots , \mu_m$, respectively. Let the spectrum of $A_{d_G}(G)$ and $A_{d_H}(H)$ be $\gamma_1,\gamma_2,\dots,\gamma_n$ and $\delta_1, \delta_2,\dots,\delta_m$ respectively, such that $\lambda_i$ and $\gamma_i$ (resp. $\mu_j$ and $\delta_j$) are the co-eigenvalues of $A(G)$ and $A_{d_G}(G)$ (resp. $A(H)$ and $A_{d_H}(H)$) for $i=1,2,\dots,n$ (resp. $j=1,2,\dots,m$). Then the AD-spectrum of $G \Box H$ is given by
			
			$$\lambda_i+\mu_j+(d_G+d_H)(\gamma_l \delta_k),
			$$ for $i,l=1,2,\dots,n$ and $j,k=1,2,\dots,m$.
		\end{thm}
		\begin{proof}
			Let $V(G)=\{u_1,u_2,\dots,u_n\}$ and $V(H)=\{v_1,v_2,\dots,v_m\}$. Let $u=(u_1,v_1)$, $v=(u_2,v_2)$ $\in V(G) \times V(H)$. Then the $(u,v)$-th entry of $AD(G \Box H)$ is given by
			\[AD(G \Box H)_{u,v}=
			\begin{cases}
				1, & \hspace{-8cm} \text{if either } d_G(u_1,u_2)=1 \text{ and } d_H(v_1,v_2)=0 \\
				\hspace{3cm} \text{or } d_G(u_1,u_2)=0, \text{ and } d_H(v_1,v_2)=1;\\
				d_G+d_H, & \hspace{-8cm} \text{if } d_G(u_1,u_2)=d_G \text{ and } d_H(v_1,v_2)=d_H;\\
				0, & \hspace{-8cm} \text{otherwise}.
			\end{cases}
			\] 
    By a suitable arrangement of the vertices of $G \Box H$, we obtain
	\begin{equation}\label{e7.1}
	AD(G \Box H)
	= (A(G) \otimes I_m) + (I_n \otimes A(H))
	+ (d_G+d_H)\, (A_{d_G}(G) \otimes A_{d_H}(H)).
	\end{equation}

	Since $G$ and $H$ are distance regular,
	$A_{d_G}(G)$ and $A_{d_H}(H)$ can be written as 
	$A_{d_G}(G) = p_{d_G}(A(G))$ and $A_{d_H}(H) = q_{d_H}(A(H))$, 
	where $p_{d_G}(x)$ and $q_{d_H}(x)$ are polynomials of degrees $d_G$ and $d_H$, respectively. Therefore,~\eqref{e7.1} becomes
	\[
	\begin{aligned}
	AD(G \Box H) 
	&= A(G \Box H) + (d_G+d_H)\, (p_{d_G}(A(G)) \otimes q_{d_H}(A(H))).
	\end{aligned}
	\]

	Let $X_i$ and $Y_j$ be the eigenvectors corresponding to the eigenvalues 
	$\lambda_i$ of $A(G)$ and $\mu_j$ of $A(H)$, respectively. Then
	\[
	\begin{aligned}
	AD(G \Box H) \, (X_i \otimes Y_j)
	&= \Big(A(G \Box H) 
	+ (d_G+d_H)\, (p_{d_G}(A(G)) \otimes q_{d_H}(A(H)))\Big)(X_i \otimes Y_j) \\
	&= A(G \Box H)(X_i \otimes Y_j)
	+  (d_G+d_H)\, \Big(p_{d_G}(A(G)) \otimes q_{d_H}(A(H))\Big)(X_i \otimes Y_j) \\
	&= (\lambda_i + \mu_j)(X_i \otimes Y_j)
	+ (d_G+d_H)\, (p_{d_G}(\lambda_i) q_{d_H}(\mu_j)) (X_i \otimes Y_j)\\
	&= (\lambda_i + \mu_j+(d_G+d_H)(\gamma_i \delta_j)) (X_i \otimes Y_j).
	\end{aligned}
	\]
	Hence $X_i \otimes Y_j$ is an eigenvector of $AD(G \Box H)$
	corresponding to the eigenvalue $\lambda_i + \mu_j+(d_G+d_H)(\gamma_i \delta_j)$ as desired.
		\end{proof}
		
		\section*{Acknowledgment} First author is supported by The Gandhigram Rural Institute under GRI Research Fellowship.

\end{document}